\documentclass[11pt]{amsart}

\usepackage{epigamath}

\usepackage[english]{babel}

\numberwithin{equation}{section}

\usepackage[utf8]{inputenc}
\usepackage{indentfirst}
\usepackage{amsmath}
\usepackage{amsthm}
\usepackage{amssymb}
\usepackage{amsfonts}
\usepackage{microtype}
\usepackage{fancyhdr}
\usepackage{tikz-cd}
\usetikzlibrary{arrows}
\usetikzlibrary{positioning}
\usepackage{xcolor}
\usepackage[colorlinks]{hyperref}
\usepackage{mathtools}
\usepackage{thmtools}
\usepackage{extarrows}

\usepackage{import}
\usepackage{xifthen}
\usepackage{pdfpages}
\usepackage{transparent}

\usepackage{csquotes}

\newtheorem{theorem}{Theorem}[section]
\newtheorem{proposition}[theorem]{Proposition}
\newtheorem{lemma}[theorem]{Lemma}
\newtheorem{corollary}[theorem]{Corollary}

\theoremstyle{definition}
\newtheorem{definition}[theorem]{Definition}

\theoremstyle{remark}
\newtheorem{numberedexample}[theorem]{Example}
\newtheorem*{remark}{Remark}

\renewcommand{\k}{\Bbbk}
\renewcommand{\O}{\mathcal{O}}
\renewcommand{\P}{\mathbb{P}}
\newcommand{\Q}{\mathbb{Q}}

\newcommand{\Spec}{\operatorname{Spec}\,}

\newcommand{\Hom}{\operatorname{Hom}}

\newcommand{\RHom}{\operatorname{RHom}}

\newcommand{\Pic}{\operatorname{Pic}}

\newcommand{\Dbcoh}{D^b_{\!\mathrm{coh}}}

\newcommand{\Perf}{\operatorname{Perf}}
\newcommand{\Dqcoh}{D_{\operatorname{QCoh}}}

\newcommand{\mA}{\mathcal{A}}
\newcommand{\mB}{\mathcal{B}}
\newcommand{\mC}{\mathcal{C}}

\newcommand{\dual}{{\scriptstyle\vee}}
\newcommand{\iso}{\simeq}
\newcommand{\caniso}{\cong}

\newcommand{\monoarrow}{\hookrightarrow}

\newcommand{\im}{\operatorname{im}\,}

\newcommand{\supp}{\operatorname{supp}}

\newcommand{\lra}{\longrightarrow}

\usepackage{quiver}

\EpigaVolumeYear{7}{2023} \EpigaArticleNr{11} \ReceivedOn{July 24, 2021}
\InFinalFormOn{December 1, 2022}
\AcceptedOn{December 16, 2022}

\title{Stably semiorthogonally indecomposable varieties}
\titlemark{Stably semiorthogonally indecomposable varieties}

\author{Dmitrii Pirozhkov}
\address{Institut de Mathématiques de Jussieu - Paris Rive Gauche (IMJ-PRG), 
4, place Jussieu, 75252 Paris Cedex 05, France}
\email{Pirozhkov@imj-prg.fr}

\authormark{D.~Pirozhkov}

\AbstractInEnglish{A triangulated category is said to be indecomposable if it admits no nontrivial semiorthogonal decompositions. We introduce a definition of a noncommutatively stably semiorthogonally indecomposable (NSSI) scheme. This property implies, among other things, that each connected closed subscheme whose structure sheaf is a perfect complex has indecomposable derived category of coherent sheaves and that if $Y$ is NSSI, then for any variety $X$ all semiorthogonal decompositions of $X \times Y$ are induced from decompositions of $X$. We prove that any scheme which admits a finite\footnote{Previous version of the paper said ``affine'' here, see Erratum at the end.} morphism to an abelian variety is NSSI and that the total space of a fibration over a regular NSSI base with NSSI fibers is also NSSI. We apply this indecomposability to deduce that there are no phantom subcategories in some varieties, including surfaces $C \times \mathbb{P}^1$, where $C$ is any smooth proper curve of positive genus.}



\acknowledgement{The work on this paper was partially supported by National Research University Higher School of Economics. This project has received funding from the European Research Council (ERC) under the European Union's Horizon 2020 research and innovation programme (Project HyperK---grant agreement 854361).}

\begin{document}


\maketitle

\begin{prelims}

\DisplayAbstractInEnglish

\bigskip

\DisplayKeyWords

\medskip

\DisplayMSCclass

\end{prelims}


\newpage

\setcounter{tocdepth}{1}

\tableofcontents


\section{Introduction}

The derived category of coherent sheaves on an algebraic variety is an interesting but complicated invariant. It is a triangulated category, and sometimes it can be built out of smaller triangulated categories using the notion of a semiorthogonal decomposition. One basic question is to identify which smooth proper varieties have indecomposable derived categories, \textit{i.e.}, admit no nontrivial semiorthogonal decompositions. Examples of such varieties are Calabi--Yau varieties, see \cite{bridgeland}, curves of positive genus, see \cite{okawa}, or more generally varieties with a globally generated canonical bundle, see \cite{kawatani-okawa}.

We propose a stronger notion of indecomposability for derived categories of algebraic varieties. Let $Y$ be an algebraic variety over a field $\k$. Our definition is a constraint on possible semiorthogonal decompositions for categories equipped with the action of the symmetric monoidal category $\Perf(Y)$ of perfect complexes on~$Y$. We work in the formalism of stable $\infty$-categories as in~\cite{higheralgebra}, but we mostly follow the exposition given in~\cite{noncommutativehpd}.

Roughly speaking, a $\Perf(Y)$-linear structure on a $\k$-linear stable $\infty$-category $\mathfrak{D}$ consists of an \emph{action functor} $a_{\mathfrak{D}}\colon \mathfrak{D} \otimes_\k \Perf(Y) \to \mathfrak{D}$ which is associative up to an isomorphism, together with higher associativity data. When there is no risk of confusion, for objects~$D \in \mathfrak{D}$ and~$S \in \Perf(Y)$, we denote the result of the action $a_{\mathfrak{D}}(D \otimes S) \in \mathfrak{D}$ by $D \cdot S$. For the duration of this paper, the term \emph{$\Perf(Y)$-linear category} means a $\k$-linear stable $\infty$-category with the~$\Perf(Y)$-linear structure as above. A basic example of a $\Perf(Y)$-linear category comes from geometry: if $f\colon X \to Y$ is a morphism of schemes, then $\Perf(X)$ is a $\Perf(Y)$-linear category with the action $D \cdot S := D \otimes_X f^*S$.

We need a couple of preliminary definitions.

\begin{definition}
  \label{def: closed under action}
  Let $Y$ be a scheme over the field $\k$. Let $\mathfrak{D}$ be a $\Perf(Y)$-linear category, and let $\mA \subset \mathfrak{D}$ be a strictly full triangulated $\k$-linear subcategory. For a subset $\mathcal{S} \subset \Perf(Y)$, we say that \emph{$\mA$ is closed under the action of $\mathcal{S}$} if for any object $S \in \mathcal{S}$ and any object $A \in \mA$, the result of the action $A \cdot S$ in $\mathfrak{D}$ is an object in the subcategory $\mA$.
\end{definition}

\begin{definition}[\textit{cf.}~\cite{bond-kapr}]
  Let $T$ be a triangulated category. A strictly full triangulated subcategory $\mA \subset T$ is a \emph{left admissible subcategory} if the inclusion functor $\mA \monoarrow T$ has a left adjoint functor.
\end{definition}

Now we can state the main definition of this paper. For the notion of a linear category proper over the base, see \cite[Lemma~4.7]{noncommutativehpd}. As an example, if $f\colon X \to Y$ is a proper perfect morphism of schemes, then $\Perf(X)$ is a proper $\Perf(Y)$-linear category by \cite[Lemma~4.9]{noncommutativehpd}.

\begin{definition}
  \label{def: nssi}
  Let $Y$ be a scheme over the field $\k$. We say that $Y$ is \emph{noncommutatively stably semiorthogonally indecomposable}, or NSSI for brevity, if for arbitrary choices of
  \begin{enumerate}
  \item $\mathfrak{D}$, a $\Perf(Y)$-linear category which is proper over $Y$ and has a classical generator,
  \item $\mA$, a left admissible subcategory of $\mathfrak{D}$, 
  \end{enumerate}
  the subcategory $\mA$ is closed under the action of $\Perf(Y)$ on $\mathfrak{D}$ in the sense of Definition~\ref{def: closed under action}.
\end{definition}

\begin{remark}
  The same definition may be given for an arbitrary symmetric monoidal stable $\infty$-category, not necessarily of the form $(\Perf(Y), \otimes)$. The author does not know whether this would be useful. Note that the variety $Y$ is completely determined by the monoidal structure on $\Perf(Y)$, see  \cite{balmer-reconstruction}, so the definition above is a property of the variety $Y$, not of something weaker.
\end{remark}

When $T$ is a triangulated category of perfect complexes on a scheme, a subcategory $\mA \subset T$ is left admissible if and only if there exists a semiorthogonal decomposition of $T$ which has $\mA$ as the first component; see \cite{bond-kapr}. Using this relation, in Lemma~\ref{lem: sanity check} we confirm that the condition in Definition~\ref{def: nssi} implies the indecomposability in the usual sense; \textit{i.e.}, all semiorthogonal decompositions of $\Perf(Y)$ are trivial. We also show in Example~\ref{ex: not stably indecomposable} that stable indecomposability is a strictly stronger condition than indecomposability. In particular, it implies that the derived category of sheaves on any smooth proper subvariety is also indecomposable.

The main results of this paper are the following two theorems.

\begin{theorem}[{ = Theorem~\textup{\ref{impl: main theorem}}}]
  \label{thm: main theorem}
  Let $Y$ be a scheme over a field $\k$. If\, $Y$ admits a finite\footnote{Previous version of the paper used ``affine'' here. See Erratum at the end.} morphism to an abelian variety over $\k$, then $Y$ is NSSI.
\end{theorem}

\begin{theorem}[{ = Theorem~\textup{\ref{impl: family of nssi varieties}}}]
  \label{thm: family of nssi varieties}
  Let $\pi\colon Y \to B$ be a flat proper morphism of quasi-compact separated schemes over a field $\k$. Assume that~$B$ is a regular\footnote{Previous version of the paper omitted this hypothesis. See Erratum at the end.} NSSI scheme and that for any closed point~$b \in B$, the fiber~$Y_b := \pi^{-1}(b)$ is a NSSI scheme.
  Then~$Y$ is NSSI.
\end{theorem}

We use these results to deduce the nonexistence of phantom subcategories in some varieties. Recall that a nonzero admissible subcategory $\mA \subset T$ of a triangulated category $T$ is called a \emph{phantom subcategory} if the class of any object of $\mA$ in the Grothendieck group $K_0(T)$ is zero. In general, such subcategories exist, but it is expected that they should not appear in many simple cases.

\begin{proposition}[{ = Proposition~\textup{\ref{impl: no phantoms}}}]
  \label{prop: no phantoms}
  Let $\k$ be an algebraically closed field of characteristic zero. Let $Y$ be a smooth projective variety over $\k$ which is NSSI.
  \begin{enumerate}
  \item Let~$X$ be either the projective line~$\P^1$ or a del Pezzo surface. Then there are no phantom subcategories in the derived category $\Dbcoh(X \times Y)$.
  \item Let~$\pi\colon \mathfrak{X} \to Y$ be an \'etale-locally trivial fibration with fiber~$\P^1$ or~$\P^2$. Then there are no phantom subcategories in the derived category~$\Dbcoh(\mathfrak{X})$.
  \end{enumerate}
\end{proposition}

\subsubsection*{Structure of the paper} In Section~\ref{sec: preliminaries}, we recall various notions useful in the rest of the paper and establish some basic properties of NSSI varieties. In Section~\ref{sec: topological rigidity}, we prove a general rigidity statement for admissible subcategories and deduce  Theorem~\ref{thm: main theorem} from it. In Section~\ref{sec: families of NSSI}, we prove Theorem~\ref{thm: family of nssi varieties}. Finally, in Section~\ref{sec: products}, we prove Proposition~\ref{prop: no phantoms}.

\subsubsection*{Notation and conventions} All triangulated categories and varieties in this paper are over a field $\k$. Tensor products, pullbacks, and pushforwards are assumed to be derived. For a subcategory $\mA$ and an object $F$ in some category the symbol $F \otimes \mA$ denotes the set of all tensor products $F \otimes A$ for objects $A \in \mA$.

\subsection*{Acknowledgments} I thank Alexander Perry for asking a question that grew into the current paper. I thank Alexander Kuznetsov and Shinnosuke Okawa for helpful suggestions. 

\section{Preliminaries}
\label{sec: preliminaries}

\subsection{Background on semiorthogonal decompositions and generators}
We recall the notion and basic properties of semiorthogonal decompositions from \cite{bond-kapr}. For a triangulated category $T$, a \emph{semiorthogonal decomposition} $T = \langle \mA, \mB \rangle$ is a pair of strictly full triangulated subcategories $\mA$ and $\mB$ such that for any $A \in \mA$ and any $B \in \mB$, the graded $\Hom$-space~$\RHom(B, A)$ vanishes, and such that for any object $E \in T$ there exists a distinguished triangle
\[
  B \lra E \lra A, 
\]
where $A \in \mA$ and $B \in \mB$. We call this the \emph{projection triangle} of $E$. It is unique up to a unique isomorphism.

When $T = \Perf(Y)$ is a triangulated category of perfect complexes on a scheme $Y$ over a field $\k$, in any semiorthogonal decomposition, the component $\mA \subset T$ is left admissible, \textit{i.e.}, its inclusion functor into~$\Perf(Y)$ admits a left adjoint, and similarly $\mB$ is right admissible.
The adjoint functors are called the left and the right projection functors, respectively. Moreover, if~$\mA \subset \Dbcoh(Y)$ is a left admissible subcategory, then the category
\[
  {}^\perp\mA := \{ B \in \Dbcoh(Y) \,\, | \,\, \forall A \in \mA, \,\, \RHom(B, A) = 0 \}
\]
is right admissible, and there is a semiorthogonal decomposition $\Dbcoh(Y) = \langle \mA, {}^\perp\mA \rangle$.

The derived category $\Dbcoh(Y)$ is said to be \emph{indecomposable} if for any semiorthogonal decomposition $\Dbcoh(Y) = \langle \mA, \mB \rangle$, either $\mA = 0$, or $\mB = 0$. Equivalently, the only admissible subcategories of $\Dbcoh(Y)$ are zero and the whole $\Dbcoh(Y)$.

We also need the notion of a classical generator of a category.

\begin{definition}
  Let $\mathcal{T}$ be a triangulated category. An object $G \in \mathcal{T}$ is called a \emph{classical generator} of $\mathcal{T}$ if the smallest triangulated subcategory of $\mathcal{T}$ containing $G$ and closed under taking direct summands is the whole~$\mathcal{T}$.
\end{definition}

\begin{lemma}
  \label{lem: classical generator and finite morphism}
  Let $f\colon Y^\prime \to Y$ be an affine morphism of schemes. If\, $G \in \Perf(Y)$ is a classical generator of\, $\Perf(Y)$, then its pullback~$f^*G$ is a classical generator of\, $\Perf(Y^\prime)$.
\end{lemma}

\begin{proof}
  By \cite[Theorem~2.1.2]{bondal-vandenbergh}, a perfect complex $f^*G$ is a classical generator for $\Perf(Y^\prime)$ if and only if for any nonzero $\mathcal{F} \in \Dqcoh(Y^\prime)$, the space $\Hom_{Y^\prime}(f^*G, \mathcal{F})$ is nonzero. By adjunction, there is an isomorphism
  \begin{equation}
    \label{eq: classical generator and pushforward}
    \Hom_{Y^\prime}(f^*G, \mathcal{F}) \caniso \Hom_Y(G, f_*\mathcal{F}).
  \end{equation}
  Since $f$ is an affine morphism, the pushforward $f_*\mathcal{F}$ of a nonzero object is nonzero. Then applying \cite[Theorem~2.1.2]{bondal-vandenbergh} on the variety~$Y$, we see that the right-hand space in \eqref{eq: classical generator and pushforward} is nonzero, and hence $f^*G$ is a classical generator.
\end{proof}

\begin{lemma}
  \label{lem: deep admissible containment criterion}
  Let $\mathfrak{D}$ be a triangulated category, and let $\mA \subset \mathfrak{D}$ be a left admissible subcategory. If an object~$E \in \mathfrak{D}$ is a classical generator of the subcategory~${}^\perp \mA$, then for any object~$F \in \mathfrak{D}$, the vanishing~$\RHom_{\mathfrak{D}}(E, F) = 0$ holds if and only if~$F \in \mA$.
\end{lemma}

\begin{proof}
  Since $\mA$ is a left admissible subcategory, it induces a semiorthogonal decomposition $\mathfrak{D} = \langle \mA, \mB \rangle$, where $\mB$ is the orthogonal subcategory ${}^\perp \mA$ and $\mB$ is a right admissible
  subcategory; see~\cite[Lemma~3.1]{bondal-exceptional}.
  
  Consider the projection triangle for the object $F$:
  \[
    B \lra F \lra A,
  \]
  where $B \in \mB$ and $A \in \mA$.
  An application of the functor $\RHom_{\mathfrak{D}}(E, -)$ to it results in a distinguished triangle of graded vector spaces:
  \[
    \RHom(E, B) \lra \RHom(E, F) \lra \RHom(E, A).
  \]
  Here the rightmost object $\RHom(E, A)$ is zero since $E \in \mB = {}^\perp \mA$. Thus $\RHom(E, F) = 0$ if and only if $\RHom(E, B) = 0$. But the object $B$ lies in $\mB$, and $E$ is a classical generator of that subcategory. Since any classical generator is a generator, see \cite[Equation~(2.1)]{bondal-vandenbergh}, the vanishing of~$\RHom(E, B)$ implies that $B$ is zero, \textit{i.e.}, that the object $F \iso A$ lies in $\mA$, as claimed.
\end{proof}

\begin{lemma}
  \label{lem: generators for admissible subcategories}
  Let $\mathfrak{D}$ be a triangulated category which admits a classical generator. Let $\mA \subset \mathfrak{D}$ be a left admissible or right admissible subcategory. Then $\mA$ has a classical generator.
\end{lemma}
\begin{proof}
  Let $G \in \mathfrak{D}$ be a classical generator. For any essentially surjective functor~$\Phi\colon\mathfrak{D} \to \mathfrak{D}^\prime$ between triangulated categories, it is easy to see that $\Phi(G)$ is a classical generator in $\mathfrak{D}^\prime$. By assumption, the inclusion functor $\mA \monoarrow \mathfrak{D}$ has either a left or a right adjoint functor, which is a projection and hence surjective on objects. Thus $\mA$ has a classical generator.
\end{proof}

\begin{lemma}
  \label{lem: admissible containment criterion}
  Let $\mathfrak{D}$ be a triangulated category which admits a classical generator. Let $\mA \subset \mathfrak{D}$ be a left admissible subcategory. Then there exists an object $E \in \mathfrak{D}$ such that an object $F \in \mathfrak{D}$ lies in~$\mA$ if and only if\, $\RHom_{\mathfrak{D}}(E, F) = 0$.
\end{lemma}

\begin{proof}
  By \cite[Lemma~3.1]{bondal-exceptional}, the subcategory ${}^\perp \mA \subset \mathfrak{D}$ is right admissible. Then by Lemma~\ref{lem: generators for admissible subcategories}, it has a classical generator. So we conclude by Lemma~\ref{lem: deep admissible containment criterion}.
\end{proof}

The proof of Theorem~\ref{impl: family of nssi varieties} uses some technical properties of linear categories. We discuss them in the two subsections below.

\subsection{Linearity for subcategories closed under the action}

Let $Y$ be a scheme over the field $\k$. Let $\mathfrak{D}$ be a $\Perf(Y)$-linear category. Assume that $\mA \subset \mathfrak{D}$ is a strictly full subcategory that is closed under the action of~$\Perf(Y)$ on~$\mathfrak{D}$ in the sense of Definition~\ref{def: closed under action}. Then it would be reasonable to expect that~$\mA$ has the structure of a~$\Perf(Y)$-linear category on its own. We use this structure in the proof of Theorem~\ref{thm: family of nssi varieties}. To define the structure, we need to construct some higher associativity data on $\mA$. The next proposition shows that this can be done in a canonical way if the subcategory is left admissible (or right admissible, with an analogous proof). The result essentially follows from the definition of an action of a symmetric monoidal $\infty$-category given in \cite{higheralgebra}. Since we were not able to find this statement in the literature, we give a sketch of the proof.

\begin{proposition}
  \label{prop: induced linearity homotopically}
  Let $Y$ be a scheme over the field $\k$. Let $\mathfrak{D}$ be a~$\Perf(Y)$-linear category, and let~$\mA \subset \mathfrak{D}$ be a left admissible $\k$-linear subcategory. Assume that $\mA$ is closed under the action of\, $\Perf(Y)$ on~$\mathfrak{D}$. Then the subcategory $\mA$ is canonically equipped with the structure of a $\Perf(Y)$-linear category, and the inclusion functor $\mA \monoarrow \mathfrak{D}$ has a canonical lift to a~$\Perf(Y)$-linear functor between~$\Perf(Y)$-linear categories.
\end{proposition}

\begin{proof}[Sketch of the proof]
  The structure of a $\Perf(Y)$-action on the stable $\infty$-category $\mA$ consists of an action functor
  \(
  a_\mA\colon \mA \otimes_{\k} \Perf(Y) \to \mA
  \)
  together with higher associativity data. As a first step, we will define the action functor on $\mA$ and check that it agrees with the action functor on $\mathfrak{D}$.

  Since $\mA$ is a left admissible subcategory, the inclusion functor $\iota_\mA\colon \mA \monoarrow \mathfrak{D}$ has a left adjoint~$\Phi\colon \mathfrak{D} \to \mA$. We use the action functor $a_{\mathfrak{D}}\colon \mathfrak{D} \otimes \Perf(Y) \to \mathfrak{D}$ on $\mathfrak{D}$ to define a functor $a_\mA$ on $\mA$ as a composition:
  \[
    a_\mA\colon
    \mA \otimes_\k \Perf(Y)
    \xrightarrow{\iota_\mA \otimes \mathrm{id}}
    \mathfrak{D} \otimes_\k \Perf(Y)
    \xlongrightarrow{a_{\mathfrak{D}}}
    \mathfrak{D}
    \xlongrightarrow{\Phi}
    \mA\rlap{.}
  \]
  By definition the functor $a_\mA$ fits into the following square of functors with a natural transformation from the upper path to the lower path given by the composition of the morphism~$a_{\mathfrak{D}} \circ (\iota_\mA \otimes \mathrm{id})$ with the unit natural transformation $\eta\colon\mathrm{id}_{\mathfrak{D}} \Rightarrow  \iota_\mA \circ \Phi $:
  \begin{equation}
    \label{eq: action linearity}
    \begin{tikzcd}
      {\mA \otimes_\k \Perf(Y)} & \mA \\
      {\mathfrak{D} \otimes_\k \Perf(Y)} & {\mathfrak{D}\rlap{.}}
      \arrow["{\iota_\mA}", from=1-2, to=2-2]
      \arrow["{\iota_\mA \otimes \mathrm{id}}", from=1-1, to=2-1]
      \arrow["{a_\mA}", from=1-1, to=1-2]
      \arrow["{a_\mathfrak{D}}", from=2-1, to=2-2]
      \arrow[shorten <=6pt, shorten >=6pt, Rightarrow, from=1-2, to=2-1]
    \end{tikzcd}\end{equation}
  
  Since $\mA$ is closed under the action of $\Perf(Y)$, the image of the functor $a_{\mathfrak{D}} \circ (\iota_\mA \otimes \mathrm{id})$ lies in the subcategory~$\mA$. The unit natural transformation $\eta$ between endofunctors of $\mathfrak{D}$ is an isomorphism on each object of $\mA \subset \mathfrak{D}$ by definition. Thus the natural transformation in the square~\eqref{eq: action linearity} is an isomorphism of functors. This property is one of the ingredients for the lift of the inclusion functor $\iota_\mA\colon \mA \monoarrow \mathfrak{D}$ to a $\Perf(Y)$-linear functor.

  We also need to provide associativity morphisms for the action functor $a_\mA$. Among other things, we need a natural isomorphism of functors that makes the square below commute, where the morphism $m\colon \Perf(Y) \otimes \Perf(Y) \to \Perf(Y)$ is the tensor product of objects:
  \begin{equation}
    \label{eq: associativity of action}
    \begin{tikzcd}
      {\mA \otimes \Perf(Y) \otimes \Perf(Y)} & {\mA \otimes \Perf(Y)} \\
      {\mA \otimes \Perf(Y)} & \mA\rlap{.}
      \arrow["{\mathrm{id} \otimes m\;}", from=1-1, to=1-2]
      \arrow["{a_\mA}", from=1-2, to=2-2]
      \arrow["{a_\mA \otimes \mathrm{id}}"', from=1-1, to=2-1]
      \arrow["{a_\mA}"', from=2-1, to=2-2]
      \arrow[shorten <=6pt, shorten >=6pt, Rightarrow, 2tail reversed, from=1-2, to=2-1]
    \end{tikzcd}\end{equation}
  Note that we already have an analogous natural isomorphism for the action on the category~$\mathfrak{D}$. Using the inclusion functor $\iota_\mA$ and the projection functor $\Phi$, we can induce the natural transformation as follows: 
  \[\begin{tikzcd}
      {\mA \otimes \Perf(Y) \otimes \Perf(Y)} & {\mathfrak{D} \otimes \Perf(Y) \otimes \Perf(Y)} & {\mathfrak{D} \otimes \Perf(Y)} \\
      & {\mathfrak{D} \otimes \Perf(Y)} & \mathfrak{D} & \mA\rlap{.}
      \arrow["{\mathrm{id} \otimes m\;}", from=1-2, to=1-3]
      \arrow["{a_\mathfrak{D}}", from=1-3, to=2-3]
      \arrow["{a_\mathfrak{D} \otimes \mathrm{id}}"', from=1-2, to=2-2]
      \arrow["{a_\mathfrak{D}}"', from=2-2, to=2-3]
      \arrow[shorten <=6pt, shorten >=6pt, Rightarrow, 2tail reversed, from=1-3, to=2-2]
      \arrow["{\iota_\mA}", from=1-1, to=1-2]
      \arrow["\Phi", from=2-3, to=2-4]
    \end{tikzcd}\]

  Using the fact that $\mA$ is closed under the action of $\Perf(Y)$, we can again check that this induced natural transformation is an isomorphism of functors.
  However, this is not exactly the natural transformation we need since its source and target functors are not the same as in~\eqref{eq: associativity of action}. Still, using the inclusion functor of~$\mA$, its adjoint, and the unit natural transformation, we can construct the required natural isomorphism using vertical and horizontal compositions. Similar arguments can be used to induce the full $\Perf(Y)$-linear structure on $\mA$ and the inclusion functor.
\end{proof}

\subsection{Some properties of linear categories}
For the proof of Theorem~\ref{thm: family of nssi varieties}, we need some properties of linear categories. They are more or less formal consequences of definitions as given in \cite{higheralgebra} or \cite{noncommutativehpd}.

First, we recall the notion of a mapping object and its properties.

\begin{definition}[\textit{cf.}~{\cite{lurie-htt}}]
  \label{def: mapping object}
  Let $Y$ be a scheme, and let $\mathfrak{D}$ be a $Y$-linear category. Let $E_1, E_2 \in \mathfrak{D}$ be two objects. An object
  $\mathcal{H} \in \Perf(Y)$
  is called a \emph{mapping object} between $E_1$ and $E_2$ if it satisfies the following universal property:
  \[
    \RHom_Y(-, \mathcal{H}) \caniso \RHom_{\mathfrak{D}}((-) \cdot E_1, E_2).
  \]
  If a mapping object exists, it is unique and is denoted by $\mathcal{H}om_{\mathfrak{D}}(E_1, E_2) \in \Perf(Y)$.
\end{definition}

\begin{remark}
  In more general contexts, a mapping object is usually defined as an object in the larger category~$\Dqcoh(Y)$, not necessarily in~$\Perf(Y)$. As shown by the following lemma, in all situations we encounter in this paper, the mapping object is perfect.
\end{remark}

\begin{lemma}
  \label{lem: existence of mapping objects}
  Let $Y$ be a quasi-compact separated scheme over a field $\k$, and let $\mathfrak{D}$ be a~\mbox{$Y$-linear} category which is proper over $Y$. Then for any two objects $E_1, E_2 \in \mathfrak{D}$, the mapping object~$\mathcal{H}om(E_1, E_2) \in \Perf(Y)$ exists.
\end{lemma}
\begin{proof}
  The mapping object exists by \cite[Proposition~5.5.2.2]{lurie-htt} as an object in $\Dqcoh(Y)$, and it is a perfect object by~\cite[Lemma~4.7]{noncommutativehpd}.
\end{proof}

\begin{lemma}
  \label{lem: internal hom and orthogonals to action}
  Let $Y$ be a quasi-compact separated scheme over a field $\k$. Let $\mathfrak{D}$ be a~$Y$-linear category which is proper over $Y$. Let $E_1, E_2 \in \mathfrak{D}$ be two objects. The following conditions are equivalent:
  \begin{enumerate}
  \item $\mathcal{H}om(E_1, E_2) \in \Perf(Y)$ is a zero object.
  \item $E_2 \in \langle \Perf(Y) \cdot E_1 \rangle^\perp$. 
  \item The orbit $\Perf(Y) \cdot E_2$ lies in $E_1^\perp$.
  \end{enumerate}
\end{lemma}

\begin{proof}
  This follows directly from Definition~\ref{def: mapping object} and the adjunction isomorphism
  \[
    \RHom_{\mathfrak{D}}(F \cdot E_1, E_2) \caniso \RHom_{\mathfrak{D}}(E_1, F^\dual \cdot E_2)
  \]
  for any object $F \in \Perf(Y)$.
\end{proof}

We continue with a discussion of base change for linear categories.

\begin{definition}
  \label{def: induced functors for linear categories}
  Let $f\colon X \to Y$ be a morphism of schemes such that the pushforward $f_*$ sends perfect complexes on~$X$ to perfect complexes on $Y$. Let $\mathfrak{D}$ be a $Y$-linear category, and let $\mathfrak{D}_X := \mathfrak{D} \otimes_{\Perf(Y)} \Perf(X)$ be its base change along the morphism $f$. Then we define the pushforward and pullback functors between $\mathfrak{D}$ and $\mathfrak{D}_X$ as follows:
  
  \[
    \begin{aligned}
      &
      f_{\mathfrak{D}, *}\colon
      \mathfrak{D}_X
      \caniso
      \mathfrak{D} \otimes_{\Perf(Y)} \Perf(X)
      \xrightarrow{\mathrm{id} \otimes f_*}
      \mathfrak{D} \otimes_{\Perf(Y)} \Perf(Y)
      \caniso
      \mathfrak{D},
      \\
      &
      f_{\mathfrak{D}}^*\colon
      \mathfrak{D}
      \caniso
      \mathfrak{D} \otimes_{\Perf(Y)} \Perf(Y)
      \xrightarrow{\;\mathrm{id} \otimes f^*\!}
      \mathfrak{D} \otimes_{\Perf(Y)} \Perf(X)
      \caniso
      \mathfrak{D}_X.
    \end{aligned}
  \]
  To simplify the notation, we often refer to these functors as just $f_*$ and $f^*$ when there is no risk of confusion.
\end{definition}

\begin{remark}
  The condition that the pushforward sends perfect complexes to perfect complexes is satisfied, for example, for a flat proper morphism or for a proper morphism to a regular variety.
\end{remark}

For the future use, we record two properties of these functors.

\begin{lemma}
  \label{lem: properties of push-pull for linear categories}
  Let $f\colon X \to Y$ be a morphism  of schemes such that the pushforward $f_*$ sends perfect complexes on~$X$ to perfect complexes on $Y$. Let~$\mathfrak{D}$ be a~$Y$-linear category, and let~$\mathfrak{D}_X := \mathfrak{D} \otimes_{\Perf(Y)} \Perf(X)$ be its base change along the morphism $f$.
  \begin{enumerate}
  \item For any objects $D \in \mathfrak{D}$ and $E \in \Perf(X)$, the following variant of the projection formula holds in $\mathfrak{D}$:
    \begin{equation}
      \label{eqn: projection formula for linear categories}
      f_*(E) \cdot D \caniso f_{\mathfrak{D}, *}(E \cdot f_{\mathfrak{D}}^*(D)),
    \end{equation}
    where on the left-hand side, we use the $\Perf(Y)$-action on $\mathfrak{D}$ and on the right-hand side we use the $\Perf(X)$-action on $\mathfrak{D}_X$.
  \item If $\mA \subset \mathfrak{D}$ is a $Y$-linear left admissible subcategory and $\mA_X := \mA \otimes_{\Perf(Y)} \Perf(X) \subset \mathfrak{D}_X$ is its base change along the morphism $f$, then
    $f_{\mathfrak{D}}^*(\mA) \subset \mA_X$ and $f_{\mathfrak{D}, *}(\mA_X) \subset \mA$.
  \end{enumerate}
\end{lemma}

\begin{proof}
  This follows directly from Definition~\ref{def: induced functors for linear categories}.
\end{proof}

\begin{lemma}
  \label{lem: change of bases and base change}
  Let $f\colon X \to Y$ be a morphism  of varieties such that the pushforward $f_*$ sends perfect complexes on~$X$ to perfect complexes on $Y$. Suppose that $f$ fits into a Cartesian square of schemes 
  \[\begin{tikzcd}
      {X} & {Y} \\
      {X^\prime} & {Y^\prime}
      \arrow["{f^\prime}", from=2-1, to=2-2]
      \arrow["{f}", from=1-1, to=1-2]
      \arrow["{g}", from=1-2, to=2-2]
      \arrow[from=1-1, to=2-1]
      \arrow["\lrcorner"{very near start, rotate=0}, from=1-1, to=2-2, phantom]
    \end{tikzcd}\]
  such that the morphism $g$ is flat.

 {\samepage Let~$\mathfrak{D}$ be a~$Y$-linear category. It can be considered as a category linear over $Y^\prime$ via the morphism $g\colon Y \to Y^\prime$. Let~$\mathfrak{D}_X$ be the base change of\, $\mathfrak{D}/Y$ along the morphism $f$, and let~$\mathfrak{D}_{X^\prime}$ be the base change of\, $\mathfrak{D}/Y^\prime$ along the morphism $f^\prime$. Then 
  \begin{enumerate}
  \item there is a natural equivalence of $X^\prime$-linear  categories $\mathfrak{D}_X \caniso \mathfrak{D}_{X^\prime}$;
  \item that equivalence identifies the functors $f_{\mathfrak{D}, *}$ and $f_{\mathfrak{D}}^*$ with the functors $f_{\mathfrak{D}^\prime, *}$ and $f_{\mathfrak{D}^\prime}^*$, respectively.
  \end{enumerate}
  }
\end{lemma}

\begin{proof}
  According to \cite[Section~2.3]{noncommutativehpd}, the base change~$\mathfrak{D}_X$ of the category~$\mathfrak{D}$ is obtained as a tensor product~$\mathfrak{D} \otimes_{\Perf(Y)} \Perf(X)$. Since the morphism $g$ is flat, by \cite[Equation~(2.1)]{noncommutativehpd}, the derived fiber product $Y \times^L_{Y^\prime} X^\prime$ is isomorphic to the usual geometric fiber product $X$, which means, explicitly, that
  \[
    \Perf(X) \caniso \Perf(Y) \otimes_{\Perf(Y^\prime)} \Perf(X^\prime).
  \]
  Thus we get a natural equivalence of categories
  \begin{equation}
    \begin{aligned}
      \mathfrak{D}_X
      :=
      \mathfrak{D} \otimes_{\Perf(Y)} \Perf(X)
      \caniso
      \mathfrak{D} \otimes_{\Perf(Y)} \left(
        \Perf(Y) \otimes_{\Perf(Y^\prime)} \Perf(X^\prime)
      \right)
      \caniso \\ \caniso
      \mathfrak{D} \otimes_{\Perf(Y^\prime)} \Perf(X^\prime)
      =:
      \mathfrak{D}_{X^\prime}.
    \end{aligned}
  \end{equation}
  It is easy to check that this equivalence is $\Perf(X^\prime)$-linear and that it respects the pushforward and the pullback functors.
\end{proof}

\subsection{Basic properties of NSSI schemes}
Now we start working with the notion of NSSI schemes, introduced in Definition~\ref{def: nssi}.

\begin{lemma}
  \label{lem: sanity check}
  Let $Y$ be a connected quasi-compact separated scheme over a field $\k$. If\, $Y$ is a NSSI scheme, then $\Perf(Y)$ is an indecomposable category.
\end{lemma}

\begin{proof}
  Let $\Perf(Y) = \langle \mA, \mB \rangle$ be a semiorthogonal decomposition. Consider $\Perf(Y)$ as a category linear over itself, with the action given by tensor product. By \cite[Theorem~3.1.1]{bondal-vandenbergh}, there exists a classical generator $G \in \Perf(Y)$. Then from Definition~\ref{def: nssi}, we deduce that both~$\mA$ (directly) and~$\mB$ (as the orthogonal to $\mA$) are closed under tensor products by arbitrary objects of~$\Perf(Y)$. In particular, for any two objects $A \in \mA$ and $B \in \mB$ we get $A \otimes B \in \mA \cap \mB = 0$, which implies $\supp(A) \cap \supp(B) = \emptyset$, where $\supp$ denotes the union of set-theoretic supports of cohomology sheaves of the complex. It follows that~$\RHom(A, B) = 0$, and hence the decomposition $\langle \mA, \mB \rangle$ is bidirectionally orthogonal, which on a connected scheme implies that one of the components is zero; see~\cite[Example~3.2]{bridgeland}.
\end{proof}

\begin{lemma}
  \label{lem: finite to indecomposable}
  Let $f\colon Y^\prime \to Y$ be a finite\footnote{Previous version of the paper used ``affine'' here. See Erratum at the end.} morphism between schemes such that the pushforward $f_*(\O_{Y^\prime})$ is a perfect complex on $Y$. If\, $Y$ is a quasi-compact separated NSSI scheme, then $Y^\prime$ is also a quasi-compact separated NSSI scheme.
\end{lemma}

\begin{proof}
  Since $Y$ is a quasi-compact separated scheme, by \cite[Theorem~3.1.1]{bondal-vandenbergh}, there exists a classical generator $G \in \Perf(Y)$. The morphism $f$ is affine, so by Lemma~\ref{lem: classical generator and finite morphism}, the pullback object~$f^*G \in \Perf(Y^\prime)$ is a classical generator in $\Perf(Y^\prime)$.
  
  Let $\mathfrak{D}$ be a $\Perf(Y^\prime)$-linear category. By definition, we need to check that any admissible subcategory $\mA \subset \mathfrak{D}$ is closed under the action of $\Perf(Y^\prime)$. The category $\mathfrak{D}$ can also be considered as a $\Perf(Y)$-linear category via the pullback map $f^*\colon \Perf(Y) \to \Perf(Y^\prime)$. Note that $\mathfrak{D}$ is proper over $Y$: we only need to show that the map $f$ makes $\Perf(Y^\prime)$ into a category proper over $Y$, which is true as soon as the derived pushforward functor $f_*\colon \Dbcoh(Y^\prime) \to \Dbcoh(Y)$ sends perfect complexes on $Y^\prime$ to perfect complexes on $Y$, see \cite[Lemma~4.9]{noncommutativehpd}. This condition, in turn, is satisfied since the morphism $f$ is affine and the pushforward of the structure sheaf of $Y^\prime$ is a perfect complex on $Y$ by assumption. So we conclude that $\mathfrak{D}$ is, indeed, proper over $Y$. Now by the assumption that $Y$ is NSSI we know that the subcategory $\mA$ is closed under the action of the subset $\im(f^*\colon \Perf(Y) \to \Perf(Y^\prime))$. In particular, $\mA$ is closed under the action of the classical generator $f^*G \in \Perf(Y^\prime)$.

  Let $A \in \mA$ be an object. Consider the subcategory $\mC \subset \Perf(Y^\prime)$ consisting of those objects~$C \in \Perf(Y^\prime)$ such that the result of the action $A \cdot C$  belongs to $\mA$. It is easy to see that $\mC$ is a triangulated subcategory closed under taking direct summands. Since it contains the classical generator $f^*G$, it coincides with the whole $\Perf(Y^\prime)$. This holds for an arbitrary object~$A \in \mA$; hence $\mA$ is closed under the action of $\Perf(Y^\prime)$ on~$\mathfrak{D}$. This shows that $Y^\prime$ is a~NSSI scheme.
\end{proof}

\begin{corollary}
  \label{cor: stable implies subvarieties}
  Let $Y$ be a quasi-compact separated scheme which is NSSI. Then for any connected closed subscheme~$Y^\prime \subset Y$ such that the structure sheaf $\O_{Y^\prime} \in \Dbcoh(Y)$ is a perfect complex, the category $\Perf(Y^\prime)$ is indecomposable.
\end{corollary}

\begin{proof}
  The inclusion $Y^\prime \subset Y$ is a finite morphism. By Lemma~\ref{lem: finite to indecomposable}, this implies that $Y^\prime$ is~NSSI, and hence by Lemma~\ref{lem: sanity check}, the category~$\Perf(Y^\prime)$ is indecomposable.
\end{proof}

\begin{numberedexample}
  \label{ex: not stably indecomposable}
  Let $Y$ be a K3 surface containing a smooth rational curve~$C \subset Y$. Then the derived category~$\Dbcoh(Y)$ is indecomposable, but $Y$ is not NSSI. Indecomposability follows from~\cite[Example~3.2]{bridgeland} since the canonical bundle of $Y$ is trivial and hence by Serre duality any semiorthogonal decomposition is completely orthogonal. However, $Y$ is not NSSI by Corollary~\ref{cor: stable implies subvarieties}: the curve~$C \iso \P^1$ in $Y$ admits a nontrivial semiorthogonal decomposition~$\Dbcoh(\P^1) = \langle \O, \O(1) \rangle$.
\end{numberedexample}

\section{NSSI property via rigidity of admissible subcategories}
\label{sec: topological rigidity}

In this section, we prove Theorem~\ref{thm: main theorem}. The main ingredient in the proof is a general rigidity property for admissible subcategories. It generalizes some results by Kawatani and Okawa~\cite{kawatani-okawa}. Namely, they proved the following rigidity statement. Let $Y$ be a smooth and proper variety over the field~$\k$. Let~$\mA \subset \Dbcoh(Y)$ be a left admissible subcategory. Then by~\cite[Theorem~1.4]{kawatani-okawa}, the subcategory $\mA$ is closed under tensor products by line bundles from~$\Pic^0(Y)$. The ideas for our proofs are more or less the same as in~\cite{kawatani-okawa}, but translated to a more general setting.

Informally speaking, the following theorem shows that in sufficiently nice situations, given a family of objects in some $\k$-linear category $\mathfrak{D}$, the condition that an object lies in a fixed left admissible subcategory $\mA \subset \mathfrak{D}$ is open. To make this rigorous, we need to put some restrictions on the category and the family, as seen below.

\begin{theorem}
  \label{thm: topological rigidity of admissible}
  Let $\mathfrak{D}$ be a proper category over a field $\k$ which has a classical generator. Let~$U$ be a quasi-compact separated scheme over $\k$, and let~$\mathfrak{D}_U$ be the base change $\mathfrak{D} \boxtimes \Perf(U)$ of\, $\mathfrak{D}$ from~$\k$ to~$U$.
  Let $\mA \subset \mathfrak{D}$ be a left admissible subcategory.

  Let~$F \in \mathfrak{D}_U$ be an object. Then there exists the
  largest Zariski-open subset $U^\prime \subset U$ such that the base change of\,~$F$ to~$\mathfrak{D}_{U^\prime} := \mathfrak{D} \boxtimes \Perf(U^\prime)$ lies in the subcategory $\mA_{U^\prime} := \mA \boxtimes \Perf(U^\prime)$. Moreover, a closed point $u \in U$ lies in $U^\prime$
  if and only if the base change~$F_u \in \mathfrak{D}_u$ of the object~$F$ along the morphism~$\{ u \} \monoarrow U$ lies in $\mA_u$.
\end{theorem}

\begin{proof}
  Since $\mathfrak{D}$ has a classical generator, by Lemma~\ref{lem: admissible containment criterion}, there exists an object~$E \in \mathfrak{D}$ such that~$\mA = E^\perp$. Since~$\mathfrak{D}_U$ is a~$U$-linear category which is proper over $U$, see \cite[Lemma~4.10]{noncommutativehpd}, 
  by Lemma~\ref{lem: existence of mapping objects}, 
  the internal $\Hom$-object~$\mathcal{H}om_U(E \boxtimes \O_U, F) \in \Perf(U)$ exists. Denote it by $\mathcal{H}$.

  Let $V$ be a quasi-compact and separated scheme with a morphism $f\colon V \to U$. Let us show that the base change $f^*F \in \mathfrak{D}_V$ lies in the subcategory~$\mA_V$ if and only if the pullback~$f^*\mathcal{H} \in \Perf(V)$ is a zero object. In order to do this, note that we have the base change isomorphism, see~\cite[Lemma~2.10]{noncommutativehpd},
  \[
    f^*\mathcal{H} \caniso \mathcal{H}om_V(f^*E, f^*F).
  \]
  By Lemma~\ref{lem: internal hom and orthogonals to action}, we see that $f^*\mathcal{H} = 0$ if and only if $f^*F \in \langle \Perf(V) \cdot f^*E \rangle^\perp$. Since $V$ is quasi-compact and separated,
  by~\cite[Theorem~3.1.1]{bondal-vandenbergh}, there is a classical generator~$G_V \in \Perf(V)$, and hence $f^*\mathcal{H} = 0$ if and only if $f^*F \in \langle G_V \cdot f^*E \rangle^\perp$. Note that the object $E \in \mathfrak{D}$ is a classical generator of the subcategory~${}^\perp \mA$ by construction in Lemma~\ref{lem: admissible containment criterion}.
  Then by \cite[Lemma~2.7]{noncommutativehpd}, the object~$G_V \cdot f^*(E)$ is a classical generator for the right admissible subcategory~${}^\perp \mA \boxtimes \Perf(V)$, and by Lemma~\ref{lem: deep admissible containment criterion}, we conclude that $f^*\mathcal{H} = 0$ if and only if $f^*F \in \mA_V$.

  From the discussion above, it is clear that the largest Zariski-open subset $U^\prime \subset U$ with the property that~$F|_{U^\prime} \in \mA_{U^\prime}$ is the complement to the set-theoretic support of $\mathcal{H} \in \Perf(U)$. The last claim of the theorem also follows since the embedding of a closed point $f\colon \{ u \} \to U$ satisfies~$f^*\mathcal{H} = 0$ if and only if the point $u$ is contained in~$U^\prime = U \setminus \supp(\mathcal{H})$.
\end{proof}

Let $Y$ be a scheme over a field $\k$.
We use the notation $\Pic^0(Y)$ to refer to the connected component of the identity in the Picard scheme of $Y$ over $\k$. We assume the following condition on $Y$:
\begin{equation}
  \label{eq:existence of nice pic0}
  \text{$\Pic^0(Y)$ exists, is proper, and there exists a Poincar\'e bundle $\mathcal{P}$ on $Y \times \Pic^0(Y)$.}
\end{equation}
The properness of $\Pic^0(Y)$ implies that
the Fourier--Mukai transform along the object~$\mathcal{P}$ defines a functor $\Perf(\Pic^0(Y)) \to \Perf(Y)$. While Theorem~\ref{thm: kawatani-okawa generalized} only uses the condition~\eqref{eq:existence of nice pic0} abstractly, we note some sufficient conditions that imply~\eqref{eq:existence of nice pic0}.

\begin{theorem}[\textit{cf.}~{\cite{murre}}]
  \label{thm: existence of pic0}
  Let $Y$ be a proper geometrically normal scheme over a field $\k$. Then~$\Pic^0(Y)$ exists. If moreover $Y(\k) \neq \emptyset$, then the Poincar\'e line bundle $\mathcal{P}$ exists on~$Y \times \Pic^0(Y)$, and hence $Y$ satisfies~\eqref{eq:existence of nice pic0}.
\end{theorem}
\begin{proof}
  The existence of the scheme~$\Pic^0(Y)$ is proved in \cite{murre}.
  If~$Y(\k) \neq \emptyset$, then by, \textit{e.g.},~\cite[Theorem~2.5]{kleiman-picard}, the relative Picard functor of $Y/\k$ is a Zariski sheaf, and thus the inclusion morphism $\Pic^0(Y) \monoarrow \Pic(Y)$ corresponds to a universal line bundle $\mathcal{P}$ on $Y \times \Pic^0(Y)$.
\end{proof}

Theorem~\ref{thm: topological rigidity of admissible} gives a general infinitesimal rigidity property for admissible subcategories. In a manner conceptually similar to \cite[Theorem~1.4]{kawatani-okawa}, we can deduce from it a global result for the invariance of admissible subcategories under the action of $\Pic^0$, though since we do not assume $\k$ to algebraically closed, we have to state the result in a slightly less transparent way.

\begin{theorem}
  \label{thm: kawatani-okawa generalized}
  Let $Y$ be a proper scheme over a field $\k$ that satisfies condition~\eqref{eq:existence of nice pic0}. Let $\mathfrak{D}$ be a $Y$-linear category which is proper over $Y$ and has a classical generator. Let $\mA \subset \mathfrak{D}$ be a left admissible subcategory. Let~$\mathcal{P}$ be a Poincar\'e line bundle on $Y \times \Pic^0(Y)$, and let~$\Phi_{\mathcal{P}}\colon \Perf(\Pic^0(Y)) \to \Perf(Y)$ be the Fourier--Mukai transform along $\mathcal{P}$.
  Then $\mA$ is closed under the action of\, $\im(\Phi_{\mathcal{P}}) \subset \Perf(Y)$.
\end{theorem}
\begin{remark}
  Assume that $\Pic^0(Y)$ is smooth (\textit{e.g.}, if the characteristic is zero).
  If $p$ is a~$\k$-point of~$\Pic^0(Y)$, then~$\Phi_{\mathcal{P}}(\O_{\{ p \}}) = \mathcal{L}_p$, where~$\O_{\{p\}} \in \Perf(\Pic^0(Y))$ is a skyscraper sheaf and $\mathcal{L}_p$ is the corresponding line bundle on $Y$. Hence Theorem~\ref{thm: kawatani-okawa generalized} implies that $\mA$ is closed under the action by all elements of~$\Pic^0(Y)(\k) \subset \Perf(Y)$. When the field~$\k$ is algebraically closed, this statement is in fact equivalent to Theorem~\ref{thm: kawatani-okawa generalized}, but to work with arbitrary fields, we need a more powerful result.
\end{remark}

\begin{proof}
  Consider the base change of the $Y$-linear category $\mathfrak{D}$ to the product $Y \times \Pic^0(Y)$:

  \[\begin{tikzcd}
      {\mA_{\Pic^0(Y)} \subset \mathfrak{D}_{\Pic^0(Y)}} & {\mA \subset \mathfrak{D}} \\
      {Y \times \mathrm{Pic}^0(Y)} & Y \\
      {\mathrm{Pic}^0(Y)} & {\Spec \k\rlap{.}}
      \arrow["{\pi_Y}", from=2-1, to=2-2]
      \arrow[from=3-1, to=3-2]
      \arrow[from=2-2, to=3-2]
      \arrow[dashed, no head, from=1-2, to=2-2]
      \arrow[dashed, no head, from=1-1, to=2-1]
      \arrow["{\pi_{\Pic^0(Y)}}", from=2-1, to=3-1]
    \end{tikzcd}\]

  By Lemma~\ref{lem: generators for admissible subcategories}, we can choose some classical generator~$G_\mA \in \mA$. To prove the theorem, it is enough to show that the set of objects $\im(\Phi_{\mathcal{P}}) \cdot G_\mA \subset \mathfrak{D}$ lies in the subcategory $\mA \subset \mathfrak{D}$.   We can simplify this statement using the definition of the Fourier--Mukai transform $\Phi_{\mathcal{P}}$ and the projection formula from Lemma~\ref{lem: properties of push-pull for linear categories}:
  \[
    \begin{aligned}
      \Phi_{\mathcal{P}}(-) \cdot G_\mA & \caniso \pi_{Y *}\left(\pi_{\Pic^0(Y)}^*(-) \otimes \mathcal{P}\right) \cdot G_\mA \caniso \pi_{Y *}\left(\left(\pi_{\Pic^0(Y)}^*(-) \otimes \mathcal{P}\right) \cdot \pi_Y^*(G_\mA)\right) \caniso \\
      & \caniso \pi_{Y *}\left(\pi_{\Pic^0(Y)}^*(-) \cdot (\mathcal{P} \cdot \pi_Y^*(G_\mA))\right).
    \end{aligned}
  \]
  If the object $\mathcal{P} \cdot \pi_Y^*(G_\mA) \in \mathfrak{D}_{\Pic^0(Y)}$ lies in the subcategory $\mA_{\Pic^0(Y)}$, then we are done: the action of $\Perf(\Pic^0(Y))$ preserves $\mA_{\Pic^0(Y)}$ by definition, and the pushforward functor~$\pi_{Y *}$ sends~$\mA_{\Pic^0(Y)}$ to~$\mA$ by the second part of Lemma~\ref{lem: properties of push-pull for linear categories}, so we would conclude that~$\im(\Phi_{\mathcal{P}})\cdot G_\mA \subset \mA$. Thus, it remains to prove that~$\mathcal{P} \cdot \pi_Y^*(G_\mA) \in \mA_{\Pic^0(Y)}$.

  Since $\mathfrak{D}$ is proper over $Y$ and $Y$ is proper over the base field, by Theorem~\ref{thm: topological rigidity of admissible}, there exists the
  largest Zariski-open subset~$U \subset \Pic^0(Y)$ such that the restriction~$(\mathcal{P} \cdot \pi_Y^*(G_\mA))|_U \in \mathfrak{D}_U$ lies in the subcategory~$\mA_U \subset \mathfrak{D}_U$. Since the restriction of the object~$\mathcal{P} \cdot \pi_Y^*(G_\mA) \in \mathfrak{D}_{\Pic^0(Y)}$ to the origin $\{ \O_Y \} \subset \Pic^0(Y)$ is the object $G_\mA \in \mathfrak{D}$, which by definition lies in $\mA$, by the second part of Theorem~\ref{thm: topological rigidity of admissible}, we know that $U$ contains the origin, and in particular is not empty.

  We claim that $U$ is closed under multiplication in the sense that for any finite field extension~$L \supset \k$, the set of~$L$-points $U(L) \subset \Pic^0(Y)(L)$ is closed under multiplication. To simplify the notation, we only spell out the full argument for~$\k$-points. By the definition of the Poincar\'e line bundle~$\mathcal{P}$, the restriction of the object~$\mathcal{P} \cdot \pi_Y^*(G_\mA) \in \mathfrak{D}_{\Pic^0(Y)}$ to a point~$p \in \Pic^0(Y)(\k)$ corresponding to a line bundle~$\mathcal{L}_p \in \Perf(Y)$ is the object $\mathcal{L}_p \cdot G_\mA$ in the category $\mathfrak{D}$. Thus by the second part of Theorem~\ref{thm: topological rigidity of admissible}, it is enough to show that if \mbox{$\mathcal{L}_1$, $\mathcal{L}_2$} are two objects in $\Perf(Y)$ such that $\mathcal{L}_i \cdot G_\mA \in \mA$ for $i = 1, 2$, then $(\mathcal{L}_1 \otimes \mathcal{L}_2) \cdot G_\mA \in \mA$. Note that by the definition of $G_\mA$, we have $\mA = \langle G_\mA \rangle$, so we see that
  \[
    (\mathcal{L}_1 \otimes \mathcal{L}_2) \cdot G_\mA = \mathcal{L}_1 \cdot (\mathcal{L}_2 \cdot G_\mA) \in \mathcal{L}_1 \cdot \mA = \langle \mathcal{L}_1 \cdot G_\mA \rangle \subset \mA.
  \]
  The case of a nontrivial finite extension~$L \supset \k$ is handled similarly, using line bundles on~$Y_L$ and the fact that the base change $(G_\mA)_L \caniso G_\mA \boxtimes \O_L$ is a generator for $\mA_L$ since the morphism~$\Spec L \to \Spec \k$ is affine.

  Thus $U$ is a nonempty Zariski-open subset of $\Pic^0(Y)$ which is closed under multiplication. This implies that $U = \Pic^0(Y)$ by the standard argument (the image of the multiplication map $U \times U \to \Pic^0(Y)$ already contains all closed points); \textit{i.e.}, the object $\mathcal{P} \cdot \pi_Y^*(G_\mA)$ lies in the subcategory $\mA_{\Pic^0(Y)}$,  which is what we wanted to show.
\end{proof}

Now we deduce Theorem~\ref{thm: main theorem} from Theorem~\ref{thm: kawatani-okawa generalized}.

\begin{proposition}
  \label{prop: abelian varieties are nssi}
  Let $A$ be an abelian variety over a field $\k$. Then $A$ is a NSSI variety.
\end{proposition}

\begin{proof}
  Let $\mathfrak{D}$ be a $\Perf(A)$-linear category which is proper over $A$ and has a classical generator. Let $\mA \subset \mathfrak{D}$ be a left admissible subcategory of $\mathfrak{D}$. According to Definition~\ref{def: nssi}, we need to show that $\mA$ is closed under the action of $\Perf(A)$.
  
  Let $\hat{A}$ be the dual abelian variety.
  Consider the functor
  \[
    \Phi\colon \Perf(\hat{A}) \lra \Perf(A)
  \]
  given by the Fourier--Mukai transform with respect to the Poincar\'e line bundle $\mathcal{P}$ on $A \times \hat{A}$. Since $A$ satisfies the condition~\eqref{eq:existence of nice pic0}, by Theorem~\ref{thm: kawatani-okawa generalized}, the subcategory $\mA \subset \mathfrak{D}$ is closed under the action of $\im(\Phi) \subset \Perf(A)$. By \cite{mukai-fm}, the functor $\Phi$ is an equivalence of categories, in particular essentially surjective. Thus $\mA \subset \mathfrak{D}$ is closed under the action of $\Perf(A)$, and hence $A$ is NSSI.
\end{proof}

\begin{theorem}[{ = Theorem~\textup{\ref{thm: main theorem}}}]
  \label{impl: main theorem}
  Let $Y$ be a scheme over a field $\k$. If\, $Y$ admits a finite\footnote{Previous version of the paper used ``affine'' here. See Erratum at the end.} morphism to an abelian variety over $\k$, then $Y$ is NSSI.
\end{theorem}

\begin{remark}
  If $Y$ is a smooth projective variety and $\k$ is algebraically closed, then the condition in Theorem~\ref{impl: main theorem} is satisfied if and only if the Albanese morphism of $Y$ is a finite morphism.
\end{remark}

\begin{proof}
  Let $f\colon Y \to A$ be a finite morphism to an abelian variety. By Proposition~\ref{prop: abelian varieties are nssi}, the variety $A$ is NSSI. Since $f$ is finite, Lemma~\ref{lem: finite to indecomposable} implies that $Y$ is also NSSI.
\end{proof}

\section{Families of NSSI schemes}
\label{sec: families of NSSI}

In this section, we show that the converse to Theorem~\ref{thm: main theorem} does not hold. In Theorem~\ref{impl: family of nssi varieties}, we give a new way to construct NSSI schemes, and the resulting schemes do not necessarily admit finite maps to abelian varieties. Namely, we show that the total space of a flat proper family of NSSI schemes over a NSSI base is itself a NSSI scheme. This can be applied, for example, to the Albanese map of a bielliptic surface: all fibers are elliptic curves, and the base of the fibration is an elliptic curve. This example is discussed in more detail in Corollary~\ref{cor: bielliptic}.

\begin{theorem}[{ = Theorem~\textup{\ref{thm: family of nssi varieties}}}]
  \label{impl: family of nssi varieties}
  Let $\pi\colon Y \to B$ be a flat proper morphism of quasi-compact separated schemes over a field $\k$. Assume that~$B$ is a regular\footnote{Previous version of the paper omitted this hypothesis. See Erratum at the end.} NSSI scheme and that for any closed point~$b \in B$, the fiber~$Y_b := \pi^{-1}(b)$ is a NSSI scheme over the field $\k(b)$.
  Then~$Y$ is NSSI.
\end{theorem}

\begin{proof}
  Let $\mathfrak{D}$ be a $Y$-linear category which is proper over $Y$ and has a classical generator. Let~$\mA \subset \mathfrak{D}$ be a left admissible subcategory. Via the morphism~$\pi\colon Y \to B$, we may consider~$\mathfrak{D}$ as a~$B$-linear category. Since~$\pi$ is a proper flat morphism, the category $\mathfrak{D}$ is proper over $B$; see \cite[Lemma~4.9]{noncommutativehpd}. By assumption, $B$ is NSSI, which implies that the left admissible subcategory $\mA \subset \mathfrak{D}$ is closed under the action of $\Perf(B)$. By Proposition~\ref{prop: induced linearity homotopically}, this implies that $\mA$ can be considered as a $B$-linear category. In particular, it is valid to consider the base change of $\mA$ along an arbitrary map to $B$.

  Pick some closed point $b \in B$, and let $j\colon Y_b \monoarrow Y$ be the fiber over that point. Consider the base change diagram:
  \[\begin{tikzcd}
      {\mA_b \subset \mathfrak{D}_b} & {\mA \subset \mathfrak{D}} \\
      {Y_b} & Y \\
      {\{ b \}} & B\rlap{.}
      \arrow[hook, from=3-1, to=3-2]
      \arrow["j", hook, from=2-1, to=2-2]
      \arrow[from=2-2, to=3-2]
      \arrow[dashed, no head, from=1-2, to=2-2]
      \arrow[dashed, no head, from=1-1, to=2-1]
      \arrow[from=2-1, to=3-1]
    \end{tikzcd}\]
  
  Here $\mA_b$ and $\mathfrak{D}_b$ are defined as base changes of $B$-linear categories $\mA$ and $\mathfrak{D}$, respectively, along the inclusion $\{ b \} \monoarrow B$. Since $B$ is assumed to be regular, the inclusion of the point $\{ b \} \monoarrow B$ is a perfect morphism, meaning that the derived pushforward sends perfect complexes to perfect complexes. A flat base change of a perfect morphism is a perfect morphism \cite[Tag~0688]{stacks-project}, so the inclusion morphism $Y_b \monoarrow Y$ also has the property that the pushforward of a perfect complex is a perfect complex on $Y$. By \cite[Lemma~3.17]{noncommutativehpd}, the subcategory $\mA_b$ is left admissible in~$\mathfrak{D}_b$.  Since
  the category $\mathfrak{D}_b$ is obtained from $\mathfrak{D}$ by the base change along an affine morphism~$\{ b \} \monoarrow B$, the pullback of the classical generator of $\mathfrak{D}$ is a classical generator of the category~$\mathfrak{D}_b$ by~\cite[Lemma~2.7]{noncommutativehpd}.

  Moreover, since $\mathfrak{D}$ is equipped with a $Y$-linear structure, by Lemma~\ref{lem: change of bases and base change}, the category~$\mathfrak{D}_b$, which as defined is only $\{ b \}$-linear, has the induced structure of a $Y_b$-linear category as the base change of $\mathfrak{D}$ along the morphism $j\colon Y_b \monoarrow Y$. Thus over $Y_b$, we have a $Y_b$-linear category $\mathfrak{D}_b$ which is proper over~$Y_b$ and has a classical generator, and a left admissible subcategory~$\mA_b \subset \mathfrak{D}_b$. By assumption, the fiber~$Y_b$ is~NSSI, so we conclude that~$\mA_b$ is closed under the action of~$\Perf(Y_b)$ on~$\mathfrak{D}_b$.

  By Lemma~\ref{lem: admissible containment criterion}, there exists an object $E \in \mathfrak{D}$ such that $E^\perp = \mA$. To show that $\mA$ is closed under the action of $\Perf(Y)$, it is enough to prove that for any $A \in \mA$, we have the inclusion~$\Perf(Y) \cdot A \in E^\perp$. By Lemma~\ref{lem: internal hom and orthogonals to action}, this is equivalent to the fact that the mapping object $\mathcal{H}om(E, A) \in \Perf(Y)$ is zero.

  Recall from Definition~\ref{def: induced functors for linear categories} that the morphism $j$ induces the pullback functor $j^*\colon \mathfrak{D} \to \mathfrak{D}_b$. We have a compatibility result for mapping objects, see \cite[Lemma~2.10]{noncommutativehpd}:
  \[
    j^*\mathcal{H}om(E, A) = \mathcal{H}om(j^*E, j^*A).
  \]
  By Definition~\ref{def: induced functors for linear categories} we know that $j^*A \in \mA_b$ and that $j^*E \in {}^\perp \mA_b$. Note that this pullback map is well-defined even if the pushforward $j_*$ does not send perfect complexes to perfect complexes, which we cannot guarantee without assuming the smoothness of $B$. The subcategory $\mA_b \subset \mathfrak{D}_b$ is closed under the action of $\Perf(Y_b)$; in particular, $\Perf(Y_b) \cdot j^*A \in \mA_b$. Thus, applying Lemma~\ref{lem: internal hom and orthogonals to action} on $Y_b$, we see that $j^*\mathcal{H}om(E, A) = 0$.

  Since $b$ is an arbitrary closed point of $B$, we conclude that the pullback of the mapping object~$\mathcal{H}om(E, A) \in \Perf(Y)$ to a fiber over any closed point of~$B$ is zero; hence~$\mathcal{H}om(E, A) = 0$ for any object $A \in \mA$. As discussed above, this implies that $\mA$ is closed under the action of~$\Perf(Y)$, and thus $Y$ is a NSSI scheme, as claimed in the statement.
\end{proof}

As an explicit application of the theorem, we show that bielliptic surfaces are NSSI varieties, even though their Albanese morphism is not finite.

\begin{corollary}
  \label{cor: bielliptic}
  Any bielliptic surface over an algebraically closed field $\k$ is a NSSI variety.
\end{corollary}


\begin{proof}
  Let $Y$ be a bielliptic surface. Recall from~\cite[Chapter~VI]{beauville} that the Albanese morphism of $Y$ is a fibration $f\colon Y \to E$ for some elliptic curve $E$, and each fiber is isomorphic to some fixed elliptic curve $F$. Since elliptic curves are NSSI by Theorem~\ref{thm: main theorem}, we can apply Theorem~\ref{thm: family of nssi varieties} to the Albanese morphism of $Y$ to conclude that $Y$ is a NSSI variety.
\end{proof}

\section{Products and fibrations}
\label{sec: products}

In this section, we discuss some applications of the NSSI property. We restrict ourselves to the class of smooth projective varieties, and in this section, we assume that $\k$ is an algebraically closed field of characteristic zero. In Lemma~\ref{lem: admissible subcategories in products}, we show that NSSI varieties are stably indecomposable in the following naive sense: if $Y$ is a~NSSI variety, then for an arbitrary smooth projective variety~$X$, any admissible subcategory in~$\Dbcoh(X \times Y)$ is induced by an admissible subcategory in~$\Dbcoh(X)$. Note that when applied to the case~$X = \{ \mathrm{pt} \}$, this result recovers Lemma~\ref{lem: sanity check}.

After that, in Proposition~\ref{impl: no phantoms}, we use the NSSI property to deduce the nonexistence of phantom subcategories in some varieties, such as total spaces of~$\P^1$- and~$\P^2$-families over~NSSI varieties. For example, this shows that there are no phantom subcategories in the surface~$\P^1 \times E$, where~$E$ is an elliptic curve.

We start with an observation about linear admissible subcategories. Given a morphism~$\pi\colon \mathfrak{X} \to U$, a~$U$-linear admissible subcategory in $\Dbcoh(\mathfrak{X})$ can be thought of as a family of admissible subcategories in the fibers of the morphism $\pi$. The moduli spaces of such fiberwise admissible subcategories have been studied in the paper~\cite{BOR}. When applied to a special situation where $\pi$ is an \'etale-locally trivial fibration, their results imply the following.

\begin{lemma}
  \label{lem: bor injection}
  Let $\pi\colon \mathfrak{X} \to U$ be a smooth and proper morphism which is an \'etale-locally trivial fibration with fiber $X$. Assume that $U$ is a connected excellent scheme over $\Q$ and that $\pi$ admits a relative ample invertible sheaf. Then for any point $u \in U$, the base change map
  \[
    \left\{
      \parbox{3.5cm}{
        \begin{center}
          \normalfont
          $U$-linear admissible \\ subcategories \\ $\mA \subset \Dbcoh(\mathfrak{X})$
        \end{center}
      }
    \right\}
    \xrightarrow{\text{\normalfont restriction to $\mathfrak{X}_u \iso X$}}
    \left\{
      \parbox{3.5cm}{
        \begin{center}
          \normalfont
          admissible \\ subcategories \\
          $\mA_X \subset \Dbcoh(X)$
        \end{center}
      }
    \right\}
  \]
  is an injection.
\end{lemma}

\begin{proof}
  By \cite[Theorem~A]{BOR}, there exists an algebraic space $\mathrm{SOD}_\pi \to U$ which is \'etale over $U$ such that the set of global sections~$\mathrm{SOD}_\pi(U)$ is the set of $U$-linear admissible subcategories of~$\Dbcoh(\mathfrak{X})$ and the fiber of~$\mathrm{SOD}_\pi$ over the point~$u \in U$ is the set of admissible subcategories in~$\Dbcoh(X)$. Since $\pi$ is assumed to be \'etale-locally trivial, the \'etale morphism $\mathrm{SOD}_\pi \to U$ is by construction also \'etale-locally trivial over~$U$, in particular separated; see \cite[Tag~02KU]{stacks-project}. A global section of a separated \'etale morphism over a connected scheme is uniquely determined by its value over a single point $u$, and this is exactly the claim of the lemma.
\end{proof}

\begin{lemma}
  \label{lem: linear subcategory in product}
  Let $X$ and $Y$ be smooth projective varieties over a field $\k$ of characteristic zero. Let $\pi\colon X \times Y \to Y$ be the projection morphism. Let $\mA \subset \Dbcoh(X \times Y)$ be an admissible subcategory. If $\mA$ is closed under the action of $\pi^*(\Dbcoh(Y)) \subset \Dbcoh(X \times Y)$, then there exists an admissible subcategory~$\mA_X \subset \Dbcoh(X)$ such that~$\mA = \mA_X \boxtimes \Dbcoh(Y)$.
\end{lemma}

\begin{proof}
  By Proposition~\ref{prop: induced linearity homotopically}, the subcategory $\mA \subset \Dbcoh(X \times Y)$ can be equipped with $Y$-linear structure. Pick a point~$y \in Y$. By Lemma~\ref{lem: bor injection}, the subcategory~$\mA$ is uniquely determined by its restriction to the fiber $\pi^{-1}(y)$, which is an admissible subcategory~$\mA_X \subset \Dbcoh(X)$. Since the base change~$\mA_X \boxtimes \Dbcoh(Y)$ is a~$Y$-linear admissible subcategory of $\Dbcoh(X \times Y)$ with the same restriction to the fiber, we conclude that~$\mA = \mA_X \boxtimes \Dbcoh(Y)$.
\end{proof}

\begin{lemma}
  \label{lem: admissible subcategories in products}
  Let $Y$ be a smooth projective variety which is NSSI. Let $X$ be any smooth projective variety, and let $\mA \subset \Dbcoh(X \times Y)$ be an admissible subcategory. Then there exists an admissible subcategory $\mA_X \subset \Dbcoh(X)$ such that~$\mA = \mA_X \boxtimes \Dbcoh(Y)$.
\end{lemma}
\begin{proof}
  Consider the projection morphism
  $
    \pi\colon X \times Y \to Y
  $.
  The category $\Dbcoh(X \times Y)$ thus may be viewed as a $Y$-linear category, with action given by the tensor product with the pullbacks along $\pi$. Since $Y$ is assumed to be a NSSI variety, the subcategory~$\mA$ is closed under the action of $\Perf(Y)$. Then we are done by Lemma~\ref{lem: linear subcategory in product}.
\end{proof}

This work started with the following application in mind. It provides many new examples of varieties which have nontrivial semiorthogonal decompositions but do not contain any phantom subcategories.

\begin{proposition}[{ = Proposition~\textup{\ref{prop: no phantoms}}}]
  \label{impl: no phantoms}
  Let $\k$ be an algebraically closed field of characteristic zero. Let $Y$ be a smooth projective variety over $\k$ which is NSSI.
  \begin{enumerate}
  \item\label{imp: no phantoms 1} Let~$X$ be either the projective line~$\P^1$ or a del Pezzo surface. Then there are no phantom subcategories in the derived category $\Dbcoh(X \times Y)$.
  \item\label{imp: no phantoms 2} Let~$\pi\colon \mathfrak{X} \to Y$ be an \'etale-locally trivial fibration with fiber~$\P^1$ or~$\P^2$. Then there are no phantom subcategories in the derived category~$\Dbcoh(\mathfrak{X})$.
  \end{enumerate}
\end{proposition}

\begin{proof}
  We deal with part~\eqref{imp: no phantoms 1} first. Let $\mA \subset \Dbcoh(X \times Y)$ be an admissible subcategory. Since~$Y$ is assumed to be NSSI, by Lemma~\ref{lem: admissible subcategories in products}, there exists an admissible subcategory~$\mA_X \subset \Dbcoh(X)$ such that $\mA = \mA_X \boxtimes \Dbcoh(Y)$. In particular, $\mA$ contains pullbacks of all objects in $\mA_X$ along the projection $p\colon X \times Y \to X$. Since the pullback map $p^*\colon K_0(X) \to K_0(X \times Y)$ is an injection,~$\mA$ is a phantom in $\Dbcoh(X \times Y)$ only if $\mA_X$ is a phantom in $\Dbcoh(X)$.

  By assumption, $X$ is either $\P^1$ or a del Pezzo surface. For $\P^1$, it is easy to show that any nontrivial admissible subcategory in $\Dbcoh(\P^1)$ is generated by a single line bundle $\O_{\P^1}(n)$, and in particular there are no phantom subcategories. If $X$ is a del Pezzo surface, then under our assumptions on the base field, $\Dbcoh(X)$ has no phantom subcategories by \cite[Theorem~6.35]{my-delpezzo}. So in both cases, there are no phantoms in~$\Dbcoh(X)$. Thus~$\mA$ is not a phantom subcategory.

  Now consider part~\eqref{imp: no phantoms 2}. Let $\mA \subset \Dbcoh(\mathfrak{X})$ be an admissible subcategory. By Definition~\ref{def: nssi} and Proposition~\ref{prop: induced linearity homotopically}, it is a $Y$-linear subcategory. Note that since the fibers of the morphism $\pi$ are projective spaces, the relative anticanonical bundle provides a relatively ample line bundle. For any \'etale-locally trivial fibration admitting a relatively ample line bundle, Lemma~\ref{lem: bor injection} shows that $\mA$ is uniquely determined by its restriction $\mA_X$ to some fiber $j\colon X \subset \mathfrak{X}$. However, since the base change procedure in Lemma~\ref{lem: bor injection} involves not only the restriction of objects from $\mA$ to the fiber, but also the closure of the resulting subcategory with respect to direct summands, it is in principle possible that a phantom subcategory in the total space may restrict to a nonphantom subcategory in the fiber.

  Another compatibility condition, proved in \cite[Theorem~5.6]{kuznetsov-basechange}, is that for any object $A \in \mA_X$, the pushforward $j_*A \in \Dbcoh(\mathfrak{X})$ lies in $\mA$. For general \'etale-locally trivial fibrations, the induced pushforward map $j_*\colon K_0(X) \to K_0(\mathfrak{X})$ does not have to be an injection. However, for the cases considered in~\eqref{imp: no phantoms 2}, it is, as we show below.

  Assume that $\pi\colon \mathfrak{X} \to Y$ is an \'etale-locally trivial fibration whose fiber is a projective space~$\P^n$. By the Blanchard--Deligne theorem~\cite[Proposition~(2.1) and Section~(2.6.2)]{deligne} the Leray--Serre spectral sequence computing the cohomology of the total space $\mathfrak{X}$ degenerates at $E_2$, and the answer is given in terms of the local systems formed by the cohomology of the fibers of the map $\pi$. Since all fibers are projective spaces, the cohomology is generated by a unique primitive ample class; hence the action of the (\'etale) fundamental group of $Y$ on the cohomology $H^*(\P^n)$ can only be trivial. Thus $H^*(\mathfrak{X})$ is abstractly isomorphic to the tensor product $H^*(Y) \otimes H^*(\P^n)$, and the Gysin morphism $j_*\colon H^*(\P^n) \to H^*(\mathfrak{X})$ is an injection. Since the normal bundle to the fiber $j\colon \P^n \monoarrow \mathfrak{X}$ is trivial, by the Grothendieck--Riemann--Roch theorem, the pushforward morphism on $K$-theory~$j_*\colon K_0(\P^n) \to K_0(\mathfrak{X})$ can be computed by taking Chern characters and applying the Gysin map; thus in particular it also is an injection.

  To sum up, any admissible subcategory $\mA \subset \Dbcoh(\mathfrak{X})$ is uniquely determined by the admissible subcategory $\mA_X \subset \Dbcoh(X)$ on some fiber. If $X$ is $\P^1$ or $\P^2$, then as explained in part~\eqref{imp: no phantoms 1}, there are no phantom subcategories in $\Dbcoh(X)$. In particular, $\mA_X$ contains some object $A \in \mA_X$ with nonzero class in $K_0(X)$. The object $j_*A$ lies in $\mA$, and as argued above, when $X$ is a projective space, the class of the pushforward in $K_0(\mathfrak{X})$ is nonzero; hence $\mA$ is not a phantom subcategory.
\end{proof}

\providecommand{\bysame}{\leavevmode\hbox to3em{\hrulefill}\thinspace}
\providecommand{\MR}{\relax\ifhmode\unskip\space\fi MR }
\providecommand{\MRhref}[2]{%
  \href{http://www.ams.org/mathscinet-getitem?mr=#1}{#2}
}
\providecommand{\href}[2]{#2}

\newpage

\section*{Erratum for the previous version of the paper}
\addcontentsline{toc}{section}{Erratum for the previous version of the paper}

The previous version of the paper erroneously stated Theorem~\ref{thm: main theorem} for schemes admitting an affine morphism into an abelian variety instead of a finite morphism. However, the key rigidity result, Theorem~\ref{thm: topological rigidity of admissible}, only works for proper categories. So only morphisms which are simultaneously affine and proper can be handled.

Philosophically the strongest possible approach would be to remove the properness condition and the existence of the classical generator from the definition of the NSSI schemes (Definition~\ref{def: nssi}) and prove this stronger property for abelian varieties. Then the careless formulation of Theorem~\ref{thm: main theorem} from the previous version of the paper would also be true. The problem is, I do not know how to prove the suggested unrestricted version of the NSSI property for any scheme which is not affine.

(I've figured out the mistake above some months ago, but procrastinated over issuing a corrected version of the paper until one of my colleagues have demonstrated that a sufficiently modern LLM can find the problem using the literal prompt "find mistakes in (arXiv link here)".)

Additionally, the previous version of the paper was too cavalier with the notion of a category proper over $\Perf(Y)$ for a scheme $Y$. A proper morphism $f\colon X \to Y$ does not, in general, realize $\Perf(X)$ as a $\Perf(Y)$-linear category proper over $Y$; this only works when the morphism is also perfect, i.e., when the pushforward of a perfect complex is perfect. For example, if $p \in Y$ is a singular point, then the inclusion $\{ p \} \monoarrow Y$ is a proper morphism, but the induced $\Perf(Y)$-linear structure on $\Dbcoh(\{ p \})$ is not proper over $Y$. When $Y$ is regular, no problems arise since any bounded object with coherent cohomology is perfect. As such, the proof of Theorem~\ref{thm: family of nssi varieties} is only valid when the base NSSI scheme $B$ is assumed to be regular.


\begin{thebibliography}{BvdB03+++}

\bibitem[Bal02]{balmer-reconstruction}
P.~Balmer, \emph{Presheaves of triangulated categories and reconstruction of
  schemes}, Math.\ Ann.\ \textbf{324} (2002), no.~3, 557--580.

\bibitem[Bea96]{beauville}
A.~Beauville, \emph{Complex algebraic surfaces}, 2nd ed., London
  Math.\ Soc.\ Stud.\ Texts, vol.~34, Cambridge Univ.\ Press,
  Cambridge, 1996.

\bibitem[BOR20]{BOR}
P.~Belmans, S.~Okawa, and A.\,T.~Ricolfi, \emph{Moduli spaces of semiorthogonal
decompositions in families}, preprint, \arXiv{2002.03303}, 2020.

\bibitem[Bon89]{bondal-exceptional}
A.\,I.~Bondal, \emph{Representations of associative algebras and coherent
sheaves}, Izv.\ Akad.\ Nauk SSSR Ser.\ Mat.\ \textbf{53} (1989), no.~1, 25--44.

\bibitem[BK90]{bond-kapr}
A.\,I.~Bondal and M.\,M.~Kapranov, \emph{Representable functors, serre functors,
  and mutations}, Mathematics of the USSR-Izvestiya \textbf{35} (1990), no.~3,
  519.

\bibitem[BVdB03]{bondal-vandenbergh}
A.~Bondal and M.~Van~den Bergh, \emph{Generators and representability of
  functors in commutative and noncommutative geometry}, Mosc.\ Math.\ J.\
\textbf{3} (2003), no.~1, 1--36, 258.


\bibitem[Bri99]{bridgeland}
T.~Bridgeland, \emph{Equivalences of triangulated categories and
  {F}ourier-{M}ukai transforms}, Bull.\ London Math.\ Soc.\ \textbf{31} (1999),
  no.~1, 25--34.


\bibitem[Del68]{deligne}
P.~Deligne, \emph{Th\'{e}or\`eme de {L}efschetz et crit\`eres de
  d\'{e}g\'{e}n\'{e}rescence de suites spectrales}, Inst.\ Hautes \'{E}tudes
  Sci.\ Publ.\ Math.\ \textbf{35} (1968), 259--278.

\bibitem[KO15]{kawatani-okawa}
K.~{Kawatani} and S.~{Okawa}, \emph{{Nonexistence of semiorthogonal
  decompositions and sections of the canonical bundle}}, preprint, \arXiv{1508.00682}, 2015.

\bibitem[Kle05]{kleiman-picard}
S.\,L.~Kleiman, \emph{The {P}icard scheme}, in: \emph{Fundamental algebraic geometry}, pp.~235--321, Math.\ Surveys Monogr., vol.~123, Amer.\ Math.\ Soc., Providence, RI, 2005.


\bibitem[Kuz11]{kuznetsov-basechange}
A.~Kuznetsov, \emph{Base change for semiorthogonal decompositions}, Compos.\
  Math.\ \textbf{147} (2011), no.~3, 852–876.

\bibitem[Lur]{higheralgebra}
J.~Lurie, \emph{Higher algebra}. Available from \url{http://www.math.harvard.edu/~lurie/}. 

\bibitem[Lur09]{lurie-htt}
\bysame, \emph{Higher topos theory}, Annals of Mathematics Studies, vol.~170,
  Princeton Univ.\ Press, Princeton, NJ, 2009.

\bibitem[Muk81]{mukai-fm}
S.~Mukai, \emph{Duality between {$D(X)$} and {$D(\hat X)$} with its application
  to {P}icard sheaves}, Nagoya Math.~J.\ \textbf{81} (1981), 153--175.

\bibitem[Mur64]{murre}
J.\,P.~Murre, \emph{On contravariant functors from the category of pre-schemes
  over a field into the category of abelian groups (with an application to the
  {P}icard functor)}, Inst.\ Hautes \'{E}tudes Sci.\ Publ.\ Math.\ \textbf{23} (1964), 5--43.

\bibitem[Oka11]{okawa}
S.~Okawa, \emph{Semi-orthogonal decomposability of the derived category of a
  curve}, Adv.\ Math.\ \textbf{228} (2011), no.~5, 2869--2873.

\bibitem[Per19]{noncommutativehpd}
A.~Perry, \emph{Noncommutative homological projective duality}, Adv.\ Math.\
  \textbf{350} (2019), 877--972.

\bibitem[Pir20]{my-delpezzo}
D.~Pirozhkov, \emph{Admissible subcategories of del pezzo surfaces}, preprint, \arXiv{2006.07643}, 2020.

\bibitem[{Sta}15]{stacks-project}
The Stacks Project Authors, \emph{Stacks project},
  \url{http://stacks.math.columbia.edu}, 2015.

\end{thebibliography}
\end{document}